\documentclass[11pt,oneside,reqno]{amsart}
\usepackage{graphicx}

\graphicspath{ {figs/} }

\usepackage{pict2e}
\usepackage{amssymb}
\usepackage{amsthm}                
\usepackage[margin=1in]{geometry}  

\usepackage{epstopdf}
\usepackage{amsmath}
\usepackage{graphicx}
\usepackage{subfigure}
\usepackage{latexsym}
\usepackage{amsmath}
\usepackage{amsfonts}
\usepackage{amssymb}
\usepackage{tikz}
\usepackage{mathdots}
\usepackage{comment}
\usepackage{float}
\usepackage[mathscr]{euscript}
\usepackage{enumerate}
\usepackage{epsfig}
\usepackage { hyperref }
\usepackage { graphics }
\usepackage { graphicx }

\usepackage {eufrak}
\usepackage[utf8]{inputenc}
\usepackage{longtable}
\usepackage[lite]{amsrefs}

\renewcommand{\PrintDOI}[1]{\href{http://dx.doi.org/\detokenize{#1}}{doi: \detokenize{#1}}%
	\IfEmptyBibField{pages}{, (to appear in print)}{}}

\theoremstyle{definition}
\newtheorem{theorem}{Theorem}[section]
\newtheorem{lemma}[theorem]{Lemma}

\newtheorem{proposition}[theorem]{Proposition}

\theoremstyle{definition}
\newtheorem{definition}[theorem]{Definition}
\newtheorem{example}[theorem]{Example}
\theoremstyle{remark}
\newtheorem{remark}[theorem]{Remark}
\numberwithin{equation}{section}

\numberwithin{equation}{section}

\title[Bataineh et al. Generating set]{Generating sets of Reidemeister moves of oriented singular links and quandles}

\author{Khaled Bataineh}
\address{Jordan University of Science and Technology, Irbid, Jordan }
\email{khaledb@just.edu.jo}

\author{Mohamed Elhamdadi}
\address{University of South Florida, Tampa, USA }
\email{emohamed@mail.usf.edu}

\author{Mustafa Hajij}
\address{University of South Florida, Tampa, USA}
\email{mhajij@usf.edu}

\author{William Youmans}
\address{University of South Florida, Tampa, USA }
\email{wyoumans@mail.usf.edu}

\date{}
\subjclass[2000]{Primary 57M25}
\keywords{Generating sets of Reidemeister moves, Quandles, Oriented singular knots}
\dedicatory{}

\begin{document}
\maketitle 
	
\begin{abstract}
We give a generating set of the generalized Reidemeister moves for oriented singular links.  
We then introduce an algebraic structure arising from the axiomatization of Reidemeister moves on oriented singular knots.  We give some examples, including some non-isomorphic families of such structures over non-abelian groups.   We show that the set of colorings of a singular knot by this new structure is an invariant of oriented singular knots and use it to distinguish some singular links. 
\end{abstract}

\tableofcontents


\bigskip

\section{Introduction}
The discovery of the Jones polynomial of links \cite{Jones} generated a search which uncovered vast families of invariants of knots and links, among them the Vassiliev knot invariants \cite{Vassiliev}. The Jones polynomial and its relatives can be computed combinatorially using knot diagrams or their braid representations. A fundamental relationship between the Jones polynomial and Vassiliev invariants was established in the work of Birman and Lin \cite{birmanlin}, where they showed that Vassiliev invariants can be characterized by three axioms. Instead of focusing on a given knot, Vassiliev changed the classical approach by deciding to study the space of all knots instead. 
As a result of this work, Vassiliev  generated the theory of singular knots and their invariants and has gained considerable attention since then. 

Singular knots and their invariants have proven to be important subjects of study on their own, and many classical knot invariants have been successfully extended to singular knots. For example,  the work of Fiedler \cite{Fiedler} where the author extended the Jones and Alexander polynomials to singular knots. The colored Jones polynomial was generalized to singular knots in \cite{BEH}. Jones-type invariants
for singular links were constructed  using a Markov trace on a version of Hecke algebras in \cite{JMA}.

The main purpose of this paper is to relate the theory of quandles to the theory of singular knots. In \cite{CEHN}, the authors developed a type of involutory quandle structure called \textit{singquandles} to study non-oriented singular knots. This article solves the last open question given in \cite{CEHN} by introducing certain algebraic structures with the intent of applying them to the case of oriented singular knots.  We call these structures \textit{oriented singquandles}. We provide multiple examples of such this new algebraic structure. Finally, for the purpose of constructing the axioms of singquandles, it was necessary to construct a generating set of Reidemeister moves acting on oriented singular links, which can be found in Section \ref{reidemeistermoves}.\\
\\
\noindent
{\bf Organization.} This article is organized as follows. In Section \ref{2} we review some of the basic theory of quandles. In Section \ref{3} we introduce the notion of oriented singquandles. In Section \ref{4} we focus on oriented singquandles whose underlying structures rely on group structures and provide some applications of singquandles to singular knot theory. Finally, in Section \ref{reidemeistermoves} we detail a generating set of Reidemeister moves for oriented singular knots.  

\section{Basics of quandles}
\label{2}

Before we introduce any algebraic structures related to singular oriented links, we will need to recall the definition of a quandle and give a few examples. For a more detailed exposition on quandles, see \cites{EN, Joyce, Matveev}.
\begin{definition}\label{Qdle}
A {\it quandle} is a set $X$ with a binary operation $(a, b) \mapsto a * b$
such that the following axioms hold:

(1) For any $a \in X$,
$a* a =a$.

(2) For any $a,b \in X$, there is a unique $x \in X$ such that 
$a= x*b$.

(3) 
For any $a,b,c \in X$, we have
$ (a*b)*c=(a*c)*(b*c). $
\end{definition}
Axiom (2) of Definition~\ref{Qdle} states that for each $y \in X$, the map $*_y:X \rightarrow X$ with $*_y(x):=x*y$ is a bijection.  Its inverse will be denoted by the mapping $\overline{*}_y:X \rightarrow X$ with $\overline{*}_y(x)=x \ \overline{*}\ y$, so that $(x*y)\ \overline{*}\ y=x=(x\ \overline{*}\ y)*y.$ Below we provide some typical examples of quandles.


\begin{itemize}
\item
Any set $X$ with the operation $x*y=x$ for any $x,y \in X$ is
a quandle called the {\it trivial} quandle.

\item
A group $X=G$ with
$n$-fold conjugation
as the quandle operation: $x*y=y^{-n} x y^n$.

\item
Let $n$ be a positive integer.
For elements  
$x, y \in \mathbb{Z}_n$ (integers modulo $n$), 
define
$x \ast y = 2y-x \pmod{n}$.
Then $\ast$ defines a quandle
structure  called the {\it dihedral quandle},
  $R_n$.

\item
For any ${\mathbb{Z}}[T, T^{-1}]$-module $M$,
$x*y=Tx+(1-T)x$ where $x,y \in M$ defines an {\it  Alexander  quandle}.

\item
Let $<\;,\;>:\mathbb{R}^n \times \mathbb{R}^n \rightarrow \mathbb{R}^n$ be a symmetric bilinear form on $\mathbb{R}^n$.  Let $X$ be the subset of $\mathbb{R}^n$ consisting of vectors $x$ such that $<x,x>\ \neq 0$.  Then the operation $$x*y=\frac{2<x,y>}{<x,x>}y-x$$ defines a quandle structure on $X$.  Note that, $x*y$ is the image of $x$ under the reflection in $y$.  This quandle is called a {\it Coxeter} quandle. 
\end{itemize}
\noindent
A function $\phi: (X, *) \rightarrow  (Y,\diamond)$ is a quandle {\it homomorphism}
if $\phi(a \ast b) = \phi(a) \diamond \phi(b)$ 
for any $a, b \in X$. A bijective quandle endomorphism of $(X,*)$ is called a quandle isomorphism.  For example any map $f_{a,b}:(\mathbb{Z}_n,*) \rightarrow (\mathbb{Z}_n, *)$ defined by $f_{a,b}(x)=ax+b$ with $a, b \in \mathbb{Z}_n$ and $a$ invertible in $\mathbb{Z}_n$ is a quandle isomorphism, where $x*y=2y-x$ (see \cite{EMR} for more details). 

\section{Oriented singular knots and quandles}\label{sec3}
\label{3}
Recall that a singular link in $S^3$ is the image of a smooth immersion of $n$ circles in $S^3$ that
has finitely many double points, called singular points. An orientation of each circle induces orientation on each component of the link. This gives an oriented singular link. In this paper we will assume that any singular link is oriented, unless specified otherwise. Furthermore, we will work with singular link projections, or diagrams, which are projections of the singular link to the plane
such  that the information at each crossing
is preserved by leaving a little break in the lower strand. Two oriented singular link diagrams are considered equivalent if and only if one can obtain one from the other by a finite sequence of singular Reidemeister moves (see Figure \ref{generatingset1123}).

In the case of classical knot theory, the axiomatization of the Reidemeister moves gives rise to the definition of a quandle. One of the goals of this paper is to generalize the structure of quandles by considering singular oriented knots and links. We will call this structure an {\it oriented singquandle} (see the definition below). The axioms of oriented singquandles will be constructed using a generating set of Reidemeister moves on oriented singular links (see Figure \ref{generatingset1123}).
 A semiarc in a singular link diagram $L$ is an edge between vertices in the link $L$ considered as a $4$-valent graph. 
 
The oriented singquandle axioms are obtained by associating elements of the oriented singquandle to semiarcs in an oriented singular link
diagram and letting these elements act on each other at crossings as shown in the following figure:

\begin{figure}[H]
	\begin{center}
	\includegraphics[scale=0.65]{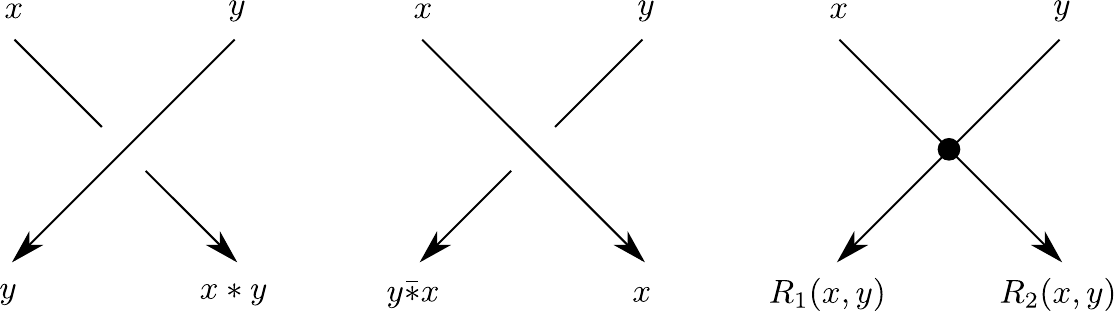}
	\end{center}
    \vspace{0.2in}
    \caption{Regular and singular crossings}
	\label{Rmoves}
\end{figure}

Now the goal is to derive the axioms that the binary operators $R_1$ and $R_2$ in Figure \ref{Rmoves} should satisfy. For this purpose, we begin with the generating set of Reidemeister moves given in Figure \ref{generatingset1123}. The proof that this is a generating set will be postponed to Section \ref{reidemeistermoves}. Using this set of Reidemeister moves, the singquandle axioms can be derived easily. This can be seen in Figures \ref{The generalized Reidemeister move RIV}, \ref{The generalized Reidemeister move RIVb} and \ref{The generalized Reidemeister move RV}.

\begin{figure}[H] 
	\begin{center}
    \includegraphics[scale=0.55]{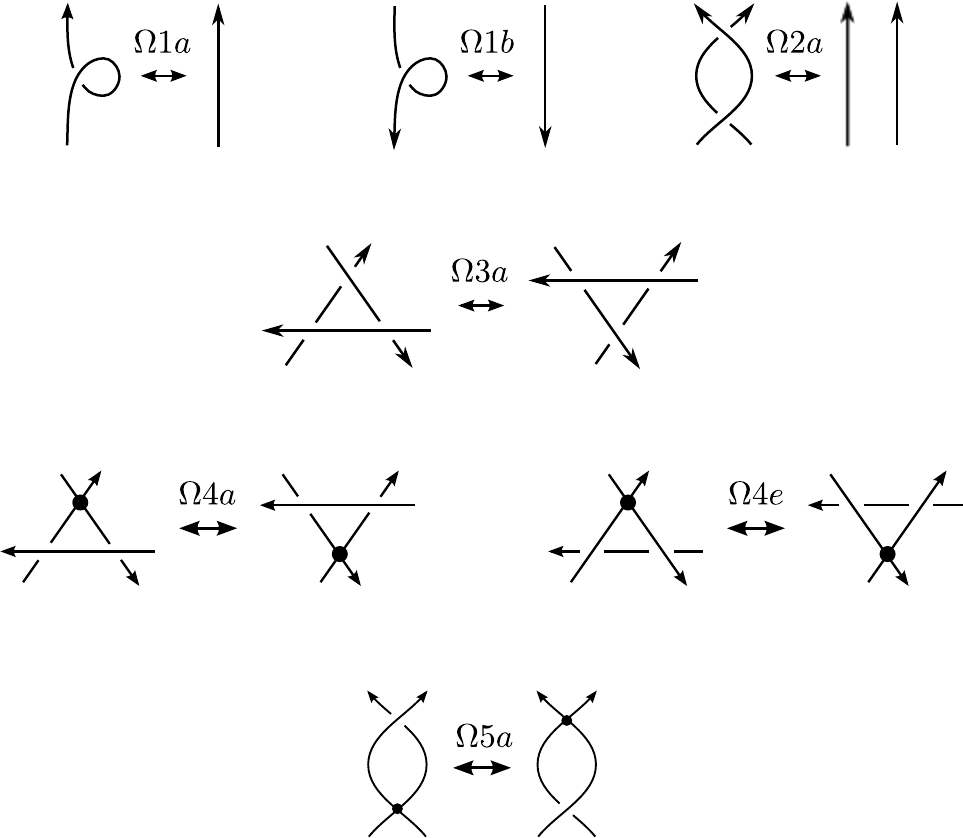}
	\vspace{.2in}
    \caption{A generating set of singular Reidemeister moves}
    \label{generatingset1123}
    \end{center}
\end{figure}

\begin{figure}[H]
	\tiny{
		\centering
		{\includegraphics[scale=0.55]{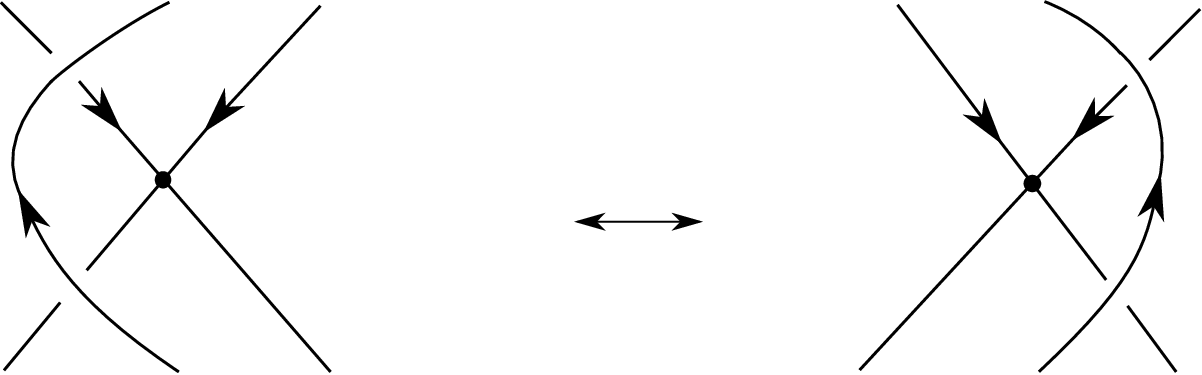}
			\put(-315,108){$x$}
			\put(-273,108){$y$}
			\put(-288,70){$x\bar{*}y$}
			\put(-344,60){$y$}
			\put(5,60){$y$}
			\put(-359,-7){$R_1(x\bar{*}y,z)*y$}
			\put(-270,-7){$y$}
			\put(-235,-7){$R_2(x\bar{*}y,z)$}
			\put(-125,-7){$R_1(x,z*y)$}
			\put(-55,-7){$y$}
			\put(-5,-7){$R_2(x,z*y)\bar{*}y$}
			\put(-238,108){$z$}
			\put(-83,108){$x$}
			\put(-38,108){$y$}
			\put(-5,108){$z$}
		}
		\vspace{.2in}
		\caption{The Reidemeister move $\Omega 4a$ and colorings} 
		\label{The generalized Reidemeister move RIV}}
\end{figure}

\begin{figure}[H]
	\tiny{
		\centering
		{\includegraphics[scale=0.55]{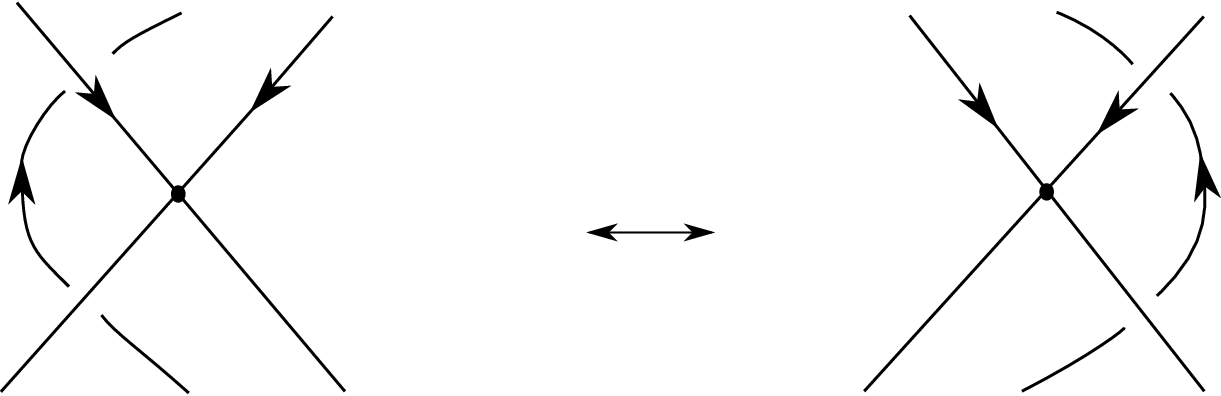}
			\put(-315,108){$x$}
			\put(-300,108){$(y\bar{*}R_1(x,z))*x$}
			\put(-288,70){$x$}
			\put(-364,60){$y\bar{*}R_1(x,z)$}
			\put(5,60){$y*R_2(x,z)$}
			\put(-359,-7){$R_1(x,z)$}
			\put(-280,-7){$y$}
			\put(-235,-7){$R_2(x,z)$}
			\put(-125,-7){$R_1(x,z)$}
			\put(-40,-7){$y$}
			\put(-5,-7){$R_2(x,z)$}
			\put(-228,108){$z$}
			\put(-83,108){$x$}
			\put(-68,108){$(y*R_2(x,z))\bar{*}z$}
			\put(5,108){$z$}
		}
		\vspace{.2in}
		\caption{The Reidemeister move $\Omega 4e$ and colorings }
		\label{The generalized Reidemeister move RIVb}}
\end{figure}

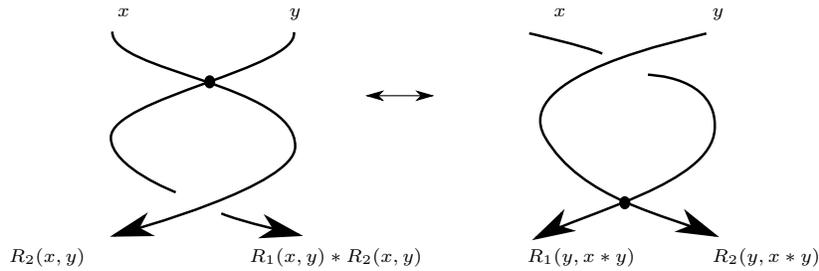
\begin{figure}[H]
\tiny{
  \centering
   {\includegraphics[scale=0.39]{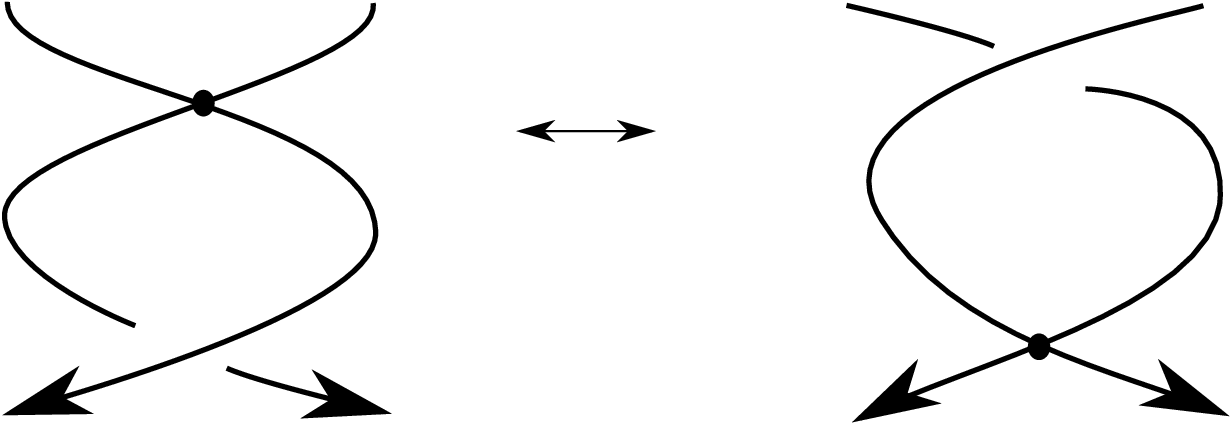}
	\put(-163,85){$y$}
    \put(-228,85){$x$}    
	\put(-3,-8){$R_2(y,x*y)$}
	 \put(-269,-8){$R_2(x,y)$}  
    \put(-63,85){$x$} 
    \put(-3,85){$y$}   
    \put(-73,-8){$R_1(y,x*y)$}
    \put(-178,-8){$R_1(x,y)*R_2(x,y)$}}
     \vspace{.2in}
     \caption{The Reidemeister move $\Omega 5a$ and colorings}
     \label{The generalized Reidemeister move RV}}
\end{figure}

The previous three figures immediately give us the following definition.
\begin{definition}\label{oriented SingQdle}
	Let $(X, *)$ be a quandle.  Let $R_1$ and $R_2$ be two maps from $X \times X$ to $X$.  The triple $(X, *, R_1, R_2)$ is called an {\it oriented singquandle} if the following axioms are satisfied:
	\begin{eqnarray}
		R_1(x\bar{*}y,z)*y&=&R_1(x,z*y) \;\;\;\text{coming from $\Omega$4a} \label{eq1}\\
		R_2(x\bar{*}y, z) & =&  R_2(x,z*y)\bar{*}y \;\;\;\text{coming from $\Omega$4a}\label{eq2}\\
	      (y\bar{*}R_1(x,z))*x   &=& (y*R_2(x,z))\bar{*}z \;\;\;\text{coming from $\Omega$4e } \label{eq3}\\
R_2(x,y)&=&R_1(y,x*y)  \;\;\;\text{coming from $\Omega$5a} \label{eq4}\\
R_1(x,y)*R_2(x,y)&=&R_2(y,x*y)  \;\;\;\text{coming from $\Omega$5a} \label{eq5}	
\end{eqnarray}	
\end{definition}
We give the following examples.
\begin{example}\label{generalizedaffinesingquandle}
Let $X=\mathbb{Z}_n$ with the quandle operation $x*y = ax+(1-a)y$, where $a$ is invertible so that $x\ \bar{*}\ y = a^{-1}x+(1-a^{-1})y$.  Now let $R_1(x,y) = bx+cy$, then by axiom~(\ref{eq4}) we have $R_2(x,y) = acx + (c(1-a) + b)y$. By substituting these expressions into axiom~(\ref{eq1}) we can find the relation $c= 1 - b$. Substituting, we find that the following is an oriented singquandle for any invertible $a$ and any $b$ in $\mathbb{Z}_n$:
	\begin{eqnarray}
    	x*y &=& ax + (1-a)y \\
        R_1(x,y) &=& bx + (1 - b)y \\
        R_2(x,y) &=& a(1 - b)x + (1 - a(1 - b))y
    \end{eqnarray}
\end{example}

It is worth noting that all of the above relations between constants can be derived from axiom 2.3, and the other axioms provide only trivial identities or the same relations. In this way, we can extend this generalized affine singquandle to the nonoriented case as well by allowing $x\ \bar{*}\ y = x*y$, since axiom 2.3 will reduce to it's counterpart in the axioms of nonoriented singquandles. (See axiom 4.1 given in \cite{CEHN}).

With this observation we can generalize the class of involutive Alexander quandles into a class of singquandles, which is given in the following example. 

\begin{example}
Let $\Lambda=\mathbb{Z}[t^{\pm 1},v]$  and let $X$ be a $\Lambda$-module. Then
the operations 
\[x\ast y=tx+(1-t)y, \quad R_1(x,y)=\alpha(a,b,c)x+(1-\alpha(a,b,c))y \quad \mathrm{and}\quad 
\]
\[R_2(x,y)=t[1  - \alpha(a,b,c)] x + [1 -t(1-  \alpha(a,b,c))] y \]
where $\alpha(a,b,c)=at+bv+ctv$, make $X$ an oriented singquandle which we call an \textit{Alexander oriented singquandle}. The fact that $X$ is an oriented singquandle follows from Example \ref{generalizedaffinesingquandle} by straightforward substitution. 
\end{example}
\begin{definition}
A \textit{coloring} of an oriented singular link $L$ is a function $C : R \longrightarrow X$, where $X$
is a fixed oriented singquandle and $R$ is the set of semiarcs in a fixed diagram of $L$, satisfying the conditions given in Figure \ref{Rmoves}. 
\end{definition}

Now the following lemma is immediate from Definition \ref{oriented SingQdle}.

\begin{lemma}
	The set of colorings of a singular link by an oriented singquandle is an invariant of oriented singular links.
\end{lemma}

 The set of colorings of a singular link $L$ by an oriented singquandle $X$ will be denoted by $Col_X(L)$.  As in the usual context of quandles, the notions of oriented singquandle homomorphisms and isomorphisms are immediate.

\begin{definition}
Let $(X,*,R_1,R_2)$ and $(Y,\triangleright,S_1,S_2)$ be two oriented singquandles. A map $f:X\longrightarrow Y$ is homomorphism if the following axioms are satisfied :
\begin{enumerate}
\item $f(x*y)=f(x)\triangleright f(y)$,
\item $f(R_1(x,y))=S_1(f(x),f(y))$,
\item $f(R_2(x,y))=S_2(f(x),f(y))$.
\end{enumerate}
If in addition $f$ is a bijection, then we say that it is an isomorphism, and we say that $(X,*,R_1,R_2)$ and $(Y,\triangleright,S_1,S_2)$ are isomorphic. Note that isomorphic oriented singquandles will induce the same set of colorings.
\end{definition}

\section{Oriented singquandles over groups}
\label{4}
When the underlying set of an oriented singquandle $X$ is a group, one obtains a rich family of structures. In this section we give a variety of examples including a generalization of affine oriented singquandles, as well as an infinite family of non-isomophic singquandles over groups.
 
\begin{example}
\label{first example}
Let $X=G$ be an abelian additive group, with $f$ being a group automorphism and $g$ a group endomorphism. Consider the operations $x*y = f(x) + y - f(y)$ and $R_1(x,y) = g(y) + x - g(x)$. First, note that the inverse operation of $*$ is given by $x\overline{*}y=f^{-1}(x)+y-f^{-1}(y) $.  Now We can deduce from axiom~(\ref{eq4}) that $R_2(x,y) = g(f(x)) + y - g(f(y))$. Substituting into the axioms, we find that the axioms are satisfied when $(f \circ g)(x) = (g \circ f)(x)$.  Example~\ref{generalizedaffinesingquandle} is a particular case of this structure with $f(x) = ax$ and $g(x) = (1-b)x$.
\end{example}

\begin{example}
	Let $X=G$ be a non-abelian multiplicative group with the quandle operation $x*y=y^{-1}xy$.  Then a direct computation gives the fact that $(X, *, R_1, R_2)$ is a singquandle if and only if $R_1$ and $R_2$ satisfy the following equations: 
	\begin{eqnarray}
		y^{-1} R_1(yxy^{-1}, z)y  &=& R_1(x, y^{-1}zy ) \\
		R_2(yxy^{-1}, z)                     & =&  y R_2(x, y^{-1}zy) y^{-1}\\
		x^{-1} R_1(x,z) y [R_1(x,z)]^{-1} x   &=& z [R_2(x, z)]^{-1}	y [R_2(x, z)] z^{-1}\\
		R_2(x, y) &=& R_1(y, y^{-1}xy)\\
       \left[ R_2(x, y) \right]^{-1} R_1(x,y) R_2(x,y) &=& R_2(y,y^{-1}xy)
	\end{eqnarray} 
\label{first example}
	A straightforward computation gives the following solutions, for all $x, y \in G$.
	\begin{enumerate}
    		\item  $R_1(x,y)=x$ and $R_2(x, y)=y$.
        \item  $R_1(x,y)=xyxy^{-1}x^{-1}$ and $R_2(x, y)=xyx^{-1}$.
        \item  $R_1(x,y)=y^{-1}xy$ and $R_2(x, y)=y^{-1}x^{-1}yxy$.
          \item  $R_1(x,y)=xy^{-1}x^{-1}yx,$ and $R_2(x, y)=x^{-1}y^{-1}xy^2$.
\item  $R_1(x,y)=y(x^{-1}y)^n$ and $R_2(x, y)=(y^{-1}x)^{n+1}y$, where $n \geq 1$.

\end{enumerate}
 \end{example}   


Next we focus our attention on a subset of some infinite families of oriented singquandle structures in order to show that some in fact are not isomorphic.
\begin{proposition}
\label{ppp}
Let $X=G$ be a non-abelian group with the binary operation $x*y=y^{-1}xy$.  Then, for $n \geq 1$, the following maps $R_1$ and $R_2$ yield an oriented singquandles structures $(X, *, R_1, R_2)$ on $G$:
\begin{enumerate}
\item
$R_1(x,y)=x(xy^{-1})^{n}$
 and $R_2(x,y)=y(x^{-1}y)^n$,
\item
$R_1(x,y)=(xy^{-1})^nx$ and 
$R_2(x,y)=(x^{-1}y)^ny,$

\item
$R_1(x,y)=x(yx^{-1})^{n+1}$ and 
$R_2(x,y)=x(y^{-1}x)^n.$

\end{enumerate}
Furthermore, in each of the cases (1), (2) and (3), different values of $n$ give also non-isomorphic singquandles.
\end{proposition}

\begin{proof}
A direct computation shows that $R_1$ and $R_2$ satisfy the five axioms of Definition ~\ref{oriented SingQdle}.
\begin{figure}[H]
\tiny{
  \centering
   {\includegraphics[scale=0.7]{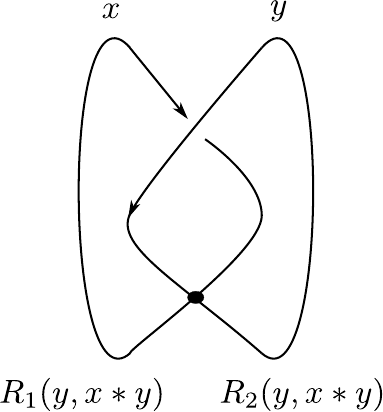}
  }\vspace{.2in}
     \caption{An oriented Hopf link}
  \label{SINGULAR_HOPF}}
\end{figure}

To see that the three solutions are pairwise non-isomorphic singquandles, we compute the set of colorings of a singular knot by each of the three solutions. Consider the singular link given in figure~\ref{SINGULAR_HOPF} and color the top arcs by elements $x$ and $y$ of $G$. Then it is easy to see that the set of colorings is given by
\begin{equation}
Col_X(L)=\{(x,y) \in  G\times G \;|\; x=R_1(y,y^{-1}xy), \; y=R_2(y,y^{-1}xy) \}.
\end{equation}
Using $R_1(x,y)=x(xy^{-1})^{n}$ and $R_2(x,y)=y(x^{-1}y)^n$ from solution (1), the set of colorings becomes: 
\begin{equation}
Col_X(L)=\{(x,y) \in  G\times G \;|\; (x^{-1}y)^{n+1}=1\},
\end{equation}
while the set of colorings of the link $L$ with $R_1(x,y)=(xy^{-1})^nx$ and 
$R_2(x,y)=(x^{-1}y)^ny$ from solution (2) is:
\begin{equation}
Col_X(L)=\{(x,y) \in  G\times G \;|\; x^{-1}(x^{-1}y)^ny=1\}.
\end{equation}
Finally, the set of colorings of the same link with $R_1(x,y)=x(yx^{-1})^{n+1}$ and 
$R_2(x,y)=x(y^{-1}x)^n$ from solution (3) is:
\begin{equation}
Col_X(L)=\{(x,y) \in  G\times G \;|\; (y^{-1}x)^n=1\}.
\end{equation}

This allows us to conclude that solutions (1), (2), and (3) are pairwise non-isomorphic oriented singquandles. In fact these computations also give that different values of $n$ in any of the three solutions (1), (2) and (3) give non-isomorphic oriented singquandles. We exclude the case of $n = 0$ as solutions (1) and (2) become equivalent.

\end{proof}

\section{A generating set of oriented singular Reidemeister moves}
\label{reidemeistermoves}

The purpose of this section is to give a generating set of oriented singular Reidemeister moves. We used this generating set to write the axioms of our singquandle structure in Section \ref{3}.

Minimal generating sets for oriented Reidemeister moves have been studied previously by Polyak in \cite{Polyak}, where he proved that $4$ moves are sufficient to generate the set of all moves for classical knots. However, analogous work for singular knot theory seems to be missing from the literature. 

It was proven in \cite{Polyak} that the moves in Figure \ref{minimal generating set of R moves} constitute a minimal generating set of Reidemeister moves on oriented classical knots. For convenience we will use in our paper the same $\Omega$ notation used by Polyak \cite{Polyak}. In this paper we will use certain moves utilized by Polyak in his paper. These moves are oriented Reidemeister moves of type 1 and type 2, see Figure \ref{oriented_R}. 
\begin{figure}[H]
	\begin{center}
	\includegraphics[scale=0.750]{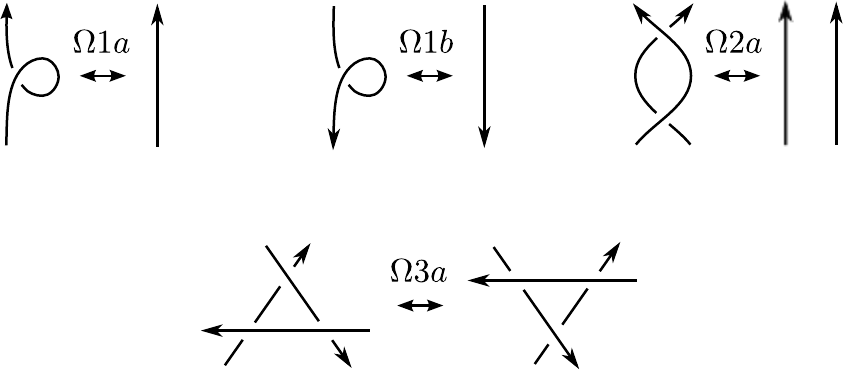}
    \vspace{0.2in}
    \caption{A generating set of Reidemeister moves for oriented knots.}
    \label{minimal generating set of R moves}
	\end{center}
\end{figure}

\begin{figure}[H]
	\begin{center}
	\includegraphics[scale=0.30]{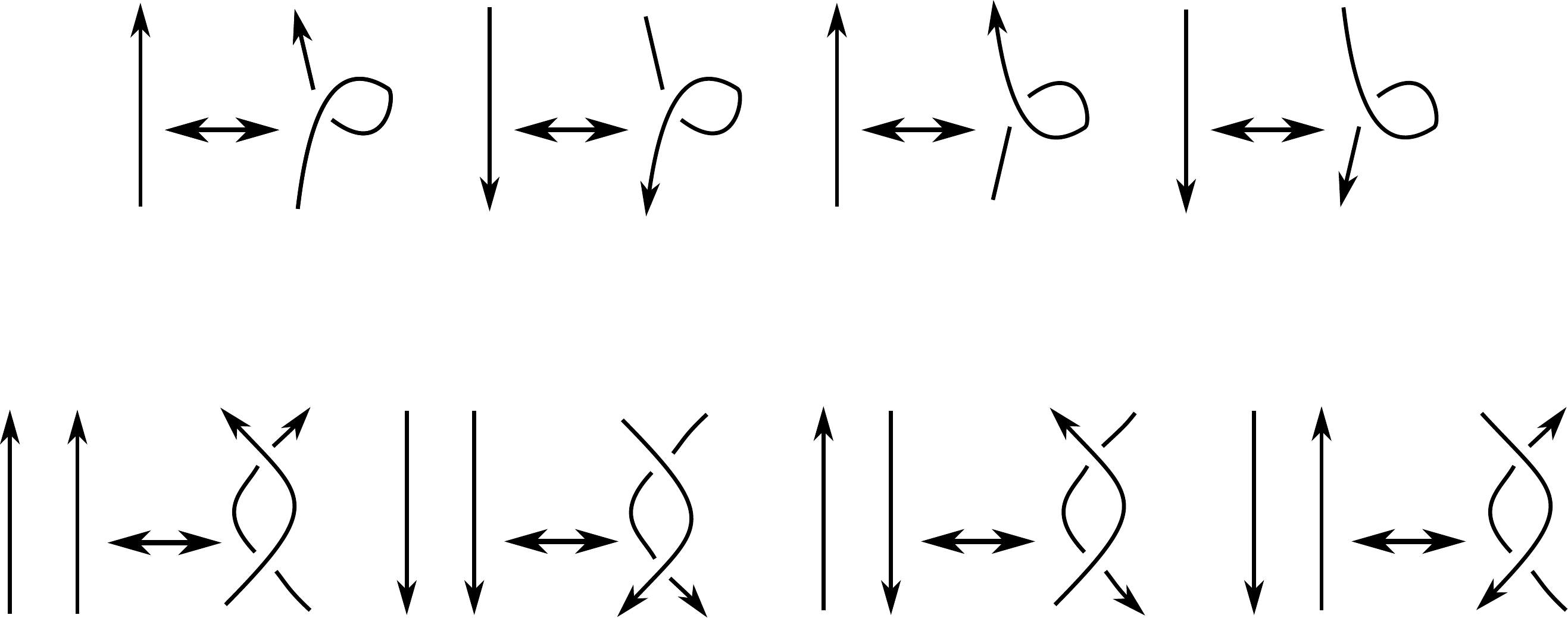}
     \put(-335,130){$\Omega1a$}
   	 \put(-250,130){$\Omega1b$}
   	 \put(-165,130){$\Omega1c$}
      \put(-80,130){$\Omega1d$}
     \put(-350,30){$\Omega2a$}
   	 \put(-255,30){$\Omega2b$}
   	 \put(-150,30){$\Omega2c$}
      \put(-50,30){$\Omega2d$}      
    \caption{Oriented Reidemeister moves of type 1 and type 2.}
    \label{oriented_R}
	\end{center}
\end{figure}

To obtain a generating set of oriented singular Reidemeister moves we enumerate all possible such moves and show that they all can be obtained by a finite sequence of the moves given in Figure \ref{generatingset1123}. Recall first that in the nonoriented case of singular knots, we have the four moves given in Figure \ref{nonoriented_here}.

\begin{figure}[h]

	\begin{center}
	\includegraphics[scale=0.80]{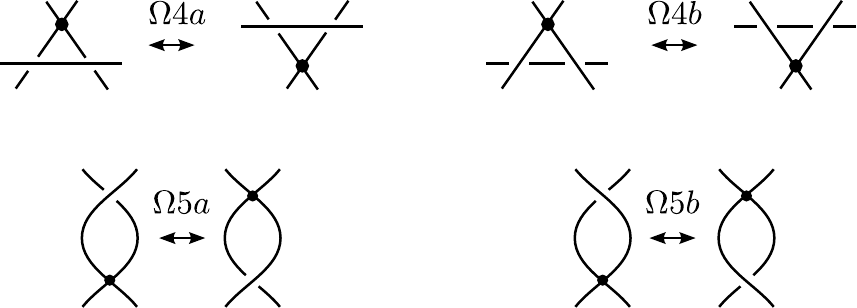}
    \vspace{0.2in}
    \caption{All four possible moves on nonoriented singular knots that involve 				 singular crossings.}
    \label{nonoriented_here}
	\end{center}
\end{figure}
If we consider the oriented case, we can easily compute the maximum number of Reidemeister moves to consider. For $\Omega 4$ moves we have 8 independent orientations, and for $\Omega 5$ we will have 6. While a combinatorically driven approach yields a handful more, it is easily seen that some moves are simply the rotation of other moves, thus they are not considered. The fourteen oriented versions for $\Omega 4$ and $\Omega 5$ moves are tabulated in Figure \ref{allmoves} for reference.

\begin{figure}[h]

	\begin{center}
	\includegraphics[scale=0.8]{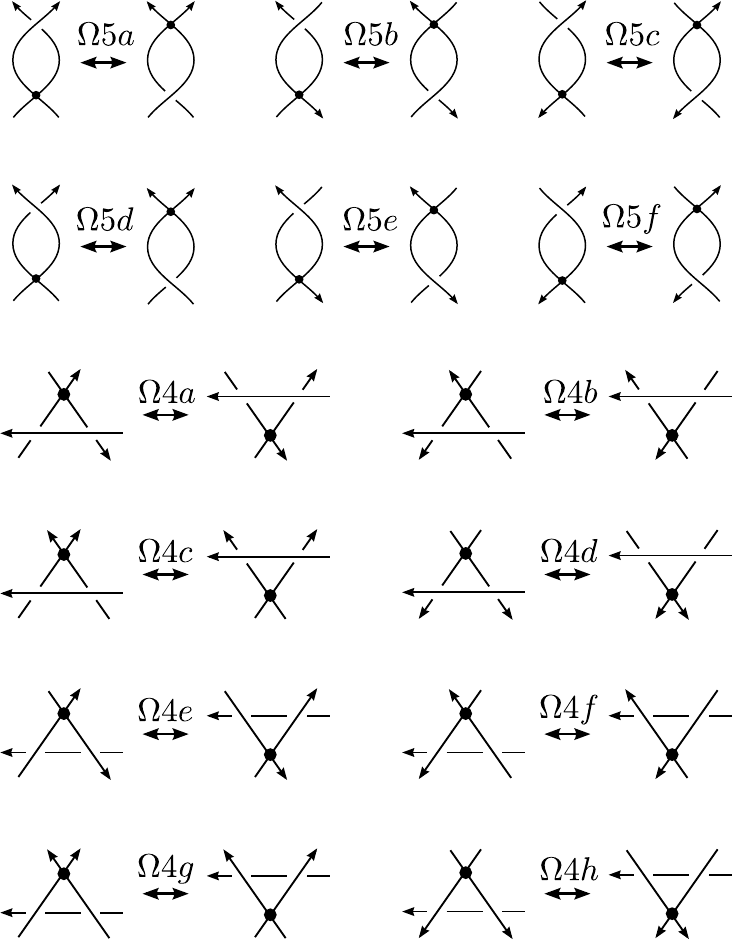}
    \vspace{0.2in}
    \caption{All fourteen oriented moves involving singular crossings.}
        \label{allmoves}
	\end{center}

\end{figure}

\begin{theorem}
Only three oriented singular Reidemeister moves are required to generate the entire set of $\Omega4$, and $\Omega5$ moves. These three moves are $\Omega4a$, $\Omega4e$, and $\Omega5a$.
\end{theorem}

To show each move's dependence on this generating set, we will need to invoke Reidemeister moves of type $\Omega1$ and $\Omega2$ while performing transformations. The specific moves are given in \cite{Polyak}. To prove this theorem, we will first show that the moves of type $\Omega4$ are generated by two unique moves. We will formulate this as a separate theorem:

\begin{theorem}
Only two oriented singular Reidemeister moves of type $\Omega4$ are required to generate all type $\Omega4$ moves. These moves are $\Omega4a$ and $\Omega4e$.
\end{theorem}
The proof of this theorem is given in Lemmas~\ref{lemma5.3} through ~\ref{lemma5.8}. 
\begin{lemma}\label{lemma5.3}
The move $\Omega4c$ is equivalent to $\Omega2c\bigcup\Omega4a\bigcup\Omega2d$.
\end{lemma}
\begin{proof}
\quad

\begin{figure}[H]
	\begin{center}
	\includegraphics[scale=0.65]{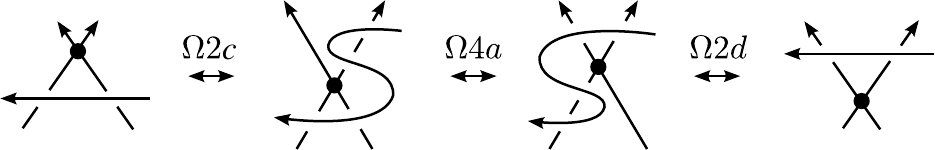}
	\end{center}
\end{figure}
\end{proof}
\begin{lemma}
\label{Omega4d}
The move $\Omega4d$ is equivalent to $\Omega2c\bigcup\Omega4a\bigcup\Omega2d$.
\end{lemma}
\begin{proof}
\quad

\begin{figure}[H]
	\begin{center}
	\includegraphics[scale=0.65]{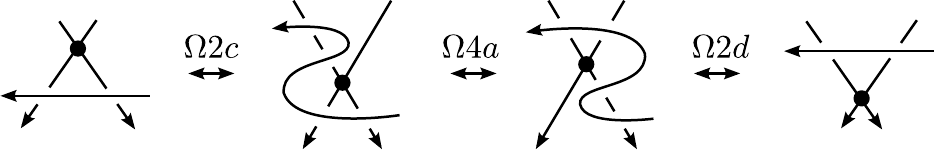}
	\end{center}
\end{figure}
\end{proof}
\begin{lemma}
The move $\Omega4g$ is equivalent to $\Omega2c\bigcup\Omega4e\bigcup\Omega2d$.
\end{lemma}
\begin{proof}
\quad

\begin{figure}[H]
	\begin{center}
	\includegraphics[scale=0.65]{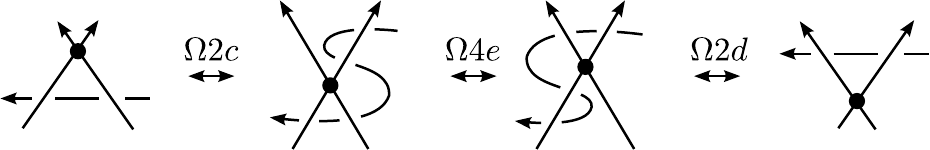}
	\end{center}
\end{figure}
\end{proof}
\begin{lemma}
\label{Omega4h}
The move $\Omega4h$ is equivalent to $\Omega2c\bigcup\Omega4e\bigcup\Omega2d$.
\end{lemma}
\begin{proof}
\quad

\begin{figure}[H]
	\begin{center}
	\includegraphics[scale=0.65]{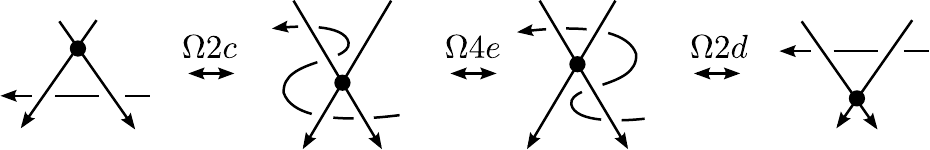}
	\end{center}
\end{figure}
\end{proof}
\begin{lemma}
The move $\Omega4b$ is equivalent to $\Omega2a\bigcup\Omega2c\bigcup\Omega4a\bigcup\Omega2d\bigcup\Omega2b$.
\end{lemma}
\begin{proof}
\quad

\begin{figure}[H]
	\begin{center}
	\includegraphics[scale=0.65]{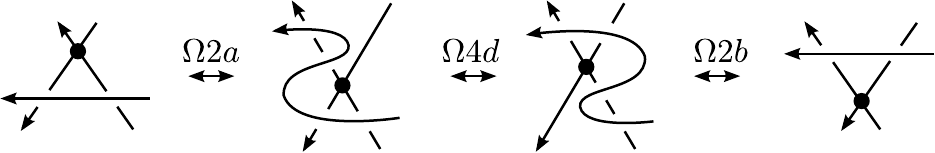}
	\end{center}
\end{figure}

By applying Lemma \ref{Omega4d} we see that the move $\Omega4d$ reduces further, and the lemma follows.
\end{proof}

\begin{lemma}\label{lemma5.8}
The move $\Omega4f$ is equivalent to $\Omega2b\bigcup\Omega2c\bigcup\Omega4e\bigcup\Omega2d\bigcup\Omega2a$.
\end{lemma}
\begin{proof}
\quad

\begin{figure}[H]
	\begin{center}
	\includegraphics[scale=0.65]{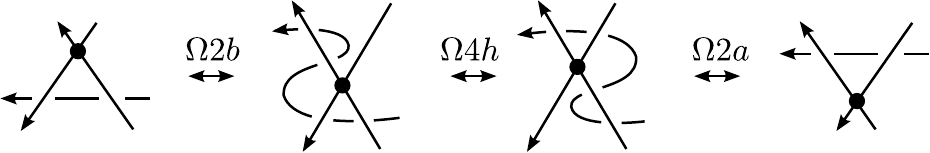}
	\end{center}
\end{figure}
By applying Lemma \ref{Omega4h} we see that the move $\Omega4h$ reduces further, and the lemma follows.
\end{proof}

From here it remains to show that all moves of type $\Omega5$ can be generated using only the $\Omega5a$ move. We formulate this as a theorem:

\begin{theorem}\label{omega5}
Only one oriented singular Reidemeister move of type $\Omega5$ is required to generate all type $\Omega5$ moves.
\end{theorem}
The proof of Theorem~\ref{omega5} is a consequence of the Lemmas~\ref{lem5.10} through~\ref{lem5.14} given below.

\begin{lemma}\label{lem5.10}
The move $\Omega5b$ is equivalent to $\Omega1a\bigcup\Omega4a\bigcup\Omega5d\bigcup\Omega4e\bigcup\Omega1a$.
\end{lemma}
\begin{proof}
\quad
\begin{figure}[H]
	\begin{center}
	\includegraphics[scale=0.65]{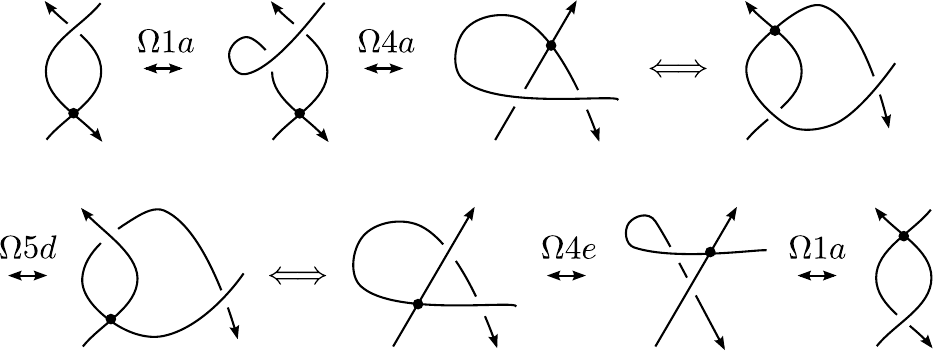}
	\end{center}
\end{figure}
\end{proof}
\begin{lemma}
The move $\Omega5c$ is equivalent to $\Omega1b\bigcup\Omega4a\bigcup\Omega5d\bigcup\Omega4e\bigcup\Omega1b$.
\end{lemma}
\begin{proof}
\quad

\begin{figure}[H]
	\begin{center}
	\includegraphics[scale=0.65]{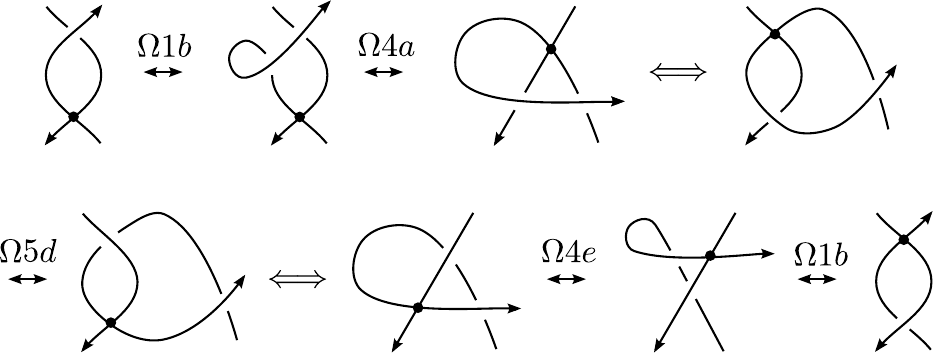}
	\end{center}
\end{figure}
\end{proof}
\begin{lemma}
The move $\Omega5e$ is equivalent to $\Omega1c\bigcup\Omega4e\bigcup\Omega5a\bigcup\Omega4a\bigcup\Omega1c$.
\end{lemma}
\begin{proof}
\quad

\begin{figure}[H]
	\begin{center}
	\includegraphics[scale=0.65]{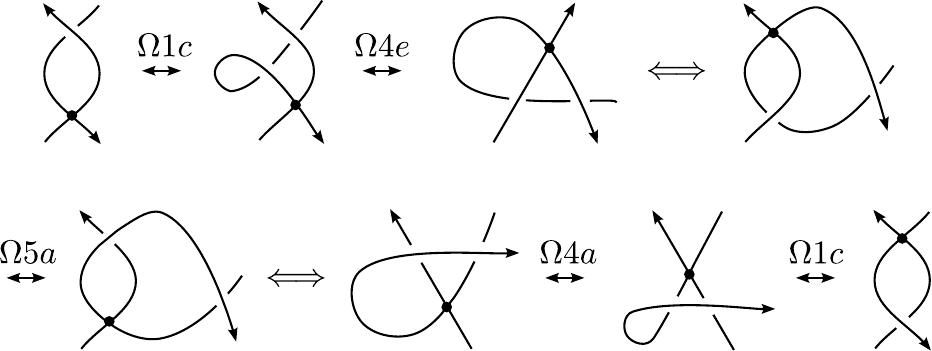}
	\end{center}
\end{figure}
\end{proof}
\begin{lemma}
The move $\Omega5f$ is equivalent to $\Omega1d\bigcup\Omega4e\bigcup\Omega5a\bigcup\Omega4a\bigcup\Omega1d$.
\end{lemma}
\begin{proof}
\quad

\begin{figure}[H]
	\begin{center}
	\includegraphics[scale=0.65]{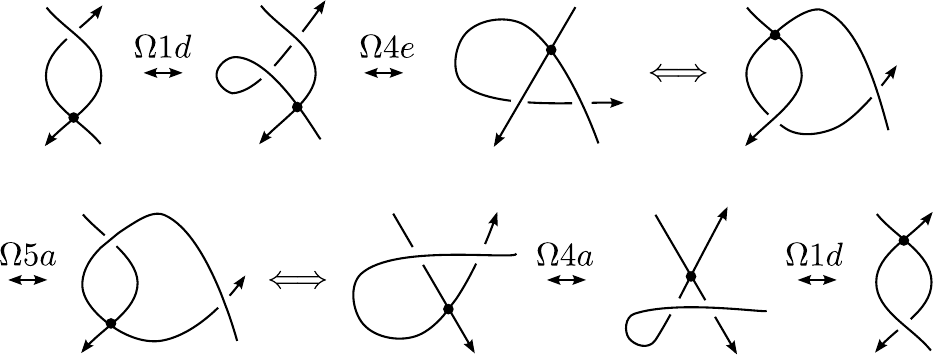}
	\end{center}
\end{figure}
\end{proof}

At this point, all $\Omega5$ moves have been shown to depend on only the moves $\Omega4a$, $\Omega4e$, $\Omega5a$, and $\Omega5d$, along with $\Omega1$ moves. The last step remaining is to eliminate $\Omega5d$.
\begin{lemma}\label{lem5.14}
The move $\Omega5d$ can be realized by a combination of $\Omega2$ moves, and one $\Omega5a$ move.
\end{lemma}
\begin{proof}
\quad
\begin{figure}[H]
	\begin{center}
    \includegraphics[scale=0.7]{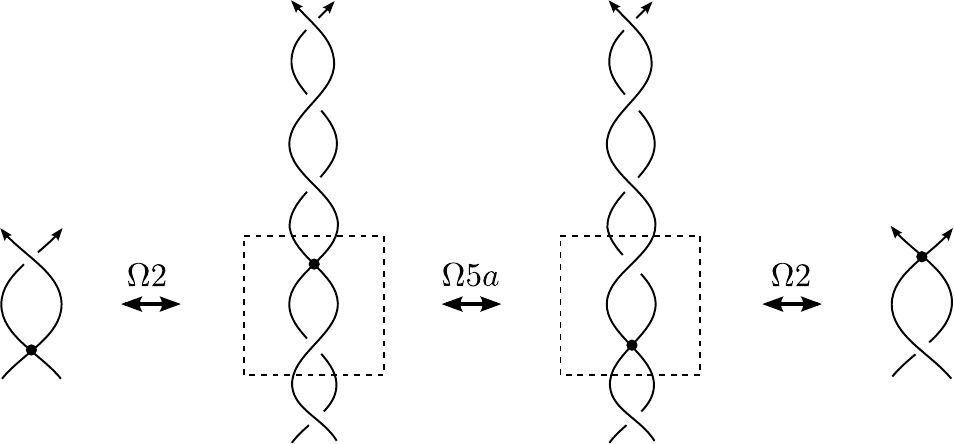}
    \end{center}
\end{figure}
\end{proof}

\begin{remark}
Recently, the problem of a generating set of oriented singular Reidemeister moves was also considered by S. Nelson, N. Oyamaguchi and R. Sazdanovic in \cite{sam} in which an alternative generating set was given.
\end{remark}

\section{Open questions}
The following are some open questions for future research:
\begin{itemize}
\item Find other generating sets of oriented singular Reidemeister moves and prove their minimality.

\item
Define a notion of {\it extensions} of oriented singquandles as in \cite{CENS}.

\item
Define a cohomology theory of oriented singquandles and use low dimensional cocycles to construct invariants that generalize the number of colorings of singular knots by oriented singquandles.

\end{itemize}

\end{document}